\documentclass[10pt]{elsarticle}
\usepackage{graphicx}
\usepackage{enumerate}
\usepackage{amsmath,amsfonts}
\usepackage{amssymb}
\usepackage{array}
\usepackage{color}
\usepackage{arydshln}
\usepackage[english]{babel}
\newcommand{\tbt}[4]{\mbox{$\left[
\begin{array}{cc}
#1 & #2 \\
#3 & #4
\end{array}
\right]$}}

\newcommand{\thbth}[9]{\mbox{$\left[
\begin{array}{ccc}
#1 & #2 & #3 \\
#4 & #5 & #6 \\
#7 & #8 & #9
\end{array}
\right]$}}

\newproof{proof}{Proof}

\newtheorem{theorem}{Theorem}[section]
\newtheorem{lemma}[theorem]{Lemma}
\newtheorem{corollary}[theorem]{Corollary}

\newtheorem{eg}[theorem]{Example}

\newtheorem{remark}[theorem]{Remark}

\begin{document}

\begin{frontmatter}

\title{On the Phases of a Complex Matrix}
\author[label2]{Dan Wang}
\ead{dwangah@connect.ust.hk}
\author[label2]{Wei Chen\corref{cor1}}
\ead{wchenust@gmail.com}
\cortext[cor1]{Corresponding author}
\author[label3]{Sei Zhen Khong}
\ead{szkhong@hku.hk}
\author[label2]{Li Qiu}
\ead{eeqiu@ust.hk}

\address[label2]{Department of Electronic and Computer Engineering, Hong Kong University of Science and Technology, Clear Water Bay, Kowloon, Hong Kong, China}
\address[label3]{Department of Electrical and Electronic Engineering, The University of Hong Kong, Pokfulam, Hong Kong, China}

\begin{abstract}
In this paper, we define the phases of a complex sectorial matrix to be its canonical angles, which are uniquely determined from a sectorial decomposition of the matrix. Various properties of matrix phases are studied, including those of compressions, Schur complements, matrix products, and sums. In particular, by exploiting a notion known as the compound numerical range, we establish a majorization relation between the phases of the eigenvalues of $AB$ and the phases of $A$ and $B$. This is then applied to investigate the rank of $I+AB$ with phase information of $A$ and $B$, which plays a crucial role in feedback stability analysis. A pair of problems: banded sectorial matrix completion and decomposition is studied. The phases of the Kronecker and Hadamard products are also discussed.
\end{abstract}

\begin{keyword}
phases \sep arguments \sep sectorial matrices \sep interlacing \sep majorization \sep rank robustness

\vskip 3pt

\MSC[2010] 15A18  \sep 15A23  \sep 15A42  \sep 15A60  \sep 15B48  \sep 15B57 \sep 15A83
\end{keyword}

\end{frontmatter}

\section{Introduction}
\label{intro}
A nonzero complex scalar $c$ can be represented in the polar form as
\[
c=\sigma e^{i\phi},
\]
where $\sigma > 0$ is called the modulus or the magnitude and $\phi$ the argument or the phase. That is,
\[
\sigma = |c| \mbox{ and } \phi = \angle c.
\]
To be more specific, the phase $\phi$ takes values in a half open $2\pi$-interval, typically $[0,2\pi)$ or $(-\pi,\pi]$.
If $c=0$, the phase $\angle c$ is undefined.

The magnitude and phase are invariant under certain operations. Specifically, the magnitude of $c$ is ``unitarily invariant'' in the sense that
$|c|=|ucv|$ for all $u,v \in \mathbb{C}$ satisfying $|u|=|v|=1$. On the other hand, the phase is ``congruence invariant'' in the sense that $\angle c = \angle w^*cw$ for all nonzero $w \in \mathbb{C}$, where the superscript $*$ denotes the complex conjugate transpose.

The magnitude and phase have the following fundamental properties:
\begin{flalign}
&&|ab| & =|a| |b| & \llap{(multiplicativity)} \label{multiplicativity}\\
&&|a+b| & \leq |a|+|b| &\llap{(subadditivity)} \label{subadditivity}\\
&&\angle (ab) & = \angle a + \angle b \ \ \ \mbox{ mod $2\pi$} &\label{multiangle}\\
&&\angle (a+b)& \in \mathrm{Co} \{ \angle a, \angle b\}  \ \ \ \mbox{if $|\angle a - \angle b| < \pi$},& \label{boundangle}
\end{flalign}
where $\mathrm{Co}$ denotes the convex hull.
The observation that (\ref{multiplicativity}) and (\ref{multiangle}) are simple equalities whereas (\ref{subadditivity}) and (\ref{boundangle}) are not enhances our understanding that the multiplication operation is easier to perform using the polar form while the addition operation the rectangular form of complex numbers.

It is well accepted that an $n \times n$ complex matrix $C$ has $n$ magnitudes, served by the $n$ singular values
\[
\sigma (C) = \begin{bmatrix} \sigma_1 (C)  & \sigma_2 (C) & \dots & \sigma_n (C) \end{bmatrix}
\]
arranged in such a way that
\[
\overline{\sigma}(C)=\sigma_1 (C) \geq \sigma_2 (C) \geq \dots \geq \sigma_n(C) = \underline{\sigma}(C).
\]
The singular values can be obtained from a singular value decomposition $C = USV^*$, where $U,V$ are unitary, and $S = \mathrm{diag}\left\{\sigma_1(C), \sigma_2(C), \dots, \sigma_n(C)\right\}$ \cite[Theorem 2.6.3]{HornJohnson}. They can also be derived from the following maximin and minimax expressions \cite[Corollary III.1.2]{Bhatia}\cite[Theorem 4.2.6]{HornJohnson}:
\begin{align*}
\sigma_i (C) &= \max _{\mathcal{M}:\mathrm{ dim }\mathcal{M}=i} \min_{x\in \mathcal{M}, \|x\|=1} \|Cx\| \\ &= \min_{\mathcal{N}:\mathrm{ dim }\mathcal{N}=n-i+1} \max_{x\in\mathcal{N}, \|x\|=1} \|Cx\|.
\end{align*}
The singular values are unitarily invariant in the sense that
\[
\sigma (C) = \sigma(U^*CV)
\]
for all unitary matrices $U$ and $V$. In particular, permuting the rows or columns of $C$
does not change its singular values.

The singular values provide a bound on the magnitudes of the eigenvalues of $C$ in the majorization order. Given two vectors $x,y\in \mathbb{R}^n$, denote by $x^\downarrow$ and $y^\downarrow$ the rearranged
versions of $x$ and $y$, respectively, in which their elements are sorted in a non-increasing order. Then, $x$ is said to be \emph{majorized} by $y$, denoted by $x\prec y$, if
\begin{align*}
\sum_{i=1}^k x^\downarrow_i\leq\sum_{i=1}^k y^\downarrow_i,\ k=1,\dots, n-1,\quad \text{and} \quad
\sum_{i=1}^n x^\downarrow_i=\sum_{i=1}^n y^\downarrow_i.
\end{align*}
Moreover, $x$ is said to be \emph{weakly majorized} by $y$ from below, denoted by $x\prec_w y$, if the last equality sign is changed to $\leq$.
For two nonnegative vectors $x,y\in\mathbb{R}^n$, $x$ is said to be \emph{log-majorized} by $y$, denoted by $x \prec_{\log} y$, if
\begin{align*}
\prod_{i=1}^k x^\downarrow_i\leq\prod_{i=1}^k y^\downarrow_i,\ k=1,\dots, n-1, \quad \text{and} \quad
\prod_{i=1}^n x^\downarrow_i=\prod_{i=1}^n y^\downarrow_i.
\end{align*}
Note that log-majorization is stronger than weak-majorization, i.e., $x \prec_{\log} y$ implies $x \prec_w y$ \cite[Chapter 5, A.2.b]{Marshall}. For more details on the theory of majorization, we refer to \cite{Marshall} and the references therein.
Denote the vector of eigenvalues of $C$ by
\[
\lambda (C) = \begin{bmatrix} \lambda_1 (C)  & \lambda_2 (C) &  \dots & \lambda_n (C) \end{bmatrix}.
\]
It is well known that the magnitudes of the eigenvalues are bounded by the singular values in the following manner \cite[Theorem 9.E.1]{Marshall}:
\[
|\lambda (C) | \prec_{\log} \sigma (C),
\]
which implies, among other things, that
\[
|\lambda (C)| \prec_w \sigma (C) .
\]



The singular values of $AB$ and $A+B$ are related to those of $A$ and $B$ through
the following majorization type inequalities resembling
(\ref{multiplicativity}) and (\ref{subadditivity}), respectively, \cite[Theorems 9.H.1 and 9.G.1.d]{Marshall}
\begin{align}
\sigma(AB)  & \prec_{\log} \sigma (A) \odot \sigma(B), \label{gainmajorization}\\
\sigma (A+B) & \prec_w \sigma(A)+\sigma(B),
\end{align}
where $\odot$ denotes the Hadamard product, i.e., the elementwise product.

In contrast to the magnitudes of a complex matrix $C$, there does not exist a universally accepted definition of phases of $C$. What properties should the phases satisfy? In this paper, we advocate a definition of matrix phases and derive some of their properties as counterparts to those of the singular values.

An early attempt \cite{Macfarlane1981} defined the phases of $C$ as the phases of the eigenvalues of the unitary part of its polar decomposition, as motivated by the seeming generalization of the polar form of a scalar to the polar decomposition of a matrix. As in the scalar case, ambiguity arises when defining phases for a singular $C$. Hence, only nonsingular matrices are relevant. More precisely, let the left and right polar decompositions of a nonsingular matrix $C$ be given by $C=PU=UQ$, where $P$ and $Q$ are positive definite and $U$ is unitary \cite[Theorem 7.3.1]{HornJohnson}. The authors of \cite{Macfarlane1981} proposed to define the phases of $C$ as
\begin{align} \label{eq: psi}
\psi (C) = \angle \lambda (U).
\end{align}
The $\angle$ function may take values in a fixed interval of length $2\pi$, such as $(-\pi, \pi]$ or $[0, 2\pi)$. 
Phases defined through the polar decomposition have the advantage that they are applicable to any square nonsingular matrices. However, they do not possess certain desired properties, which will be discussed later.

In this paper, we shall revisit the matrix canonical angles introduced in \cite{FurtadoJohnson2001} and propose to adopt them as the phases of a complex matrix. The phases defined in this way have the desired properties as shown in earlier studies and this paper.

\section{Definition of phases}\label{phase def}
First, we review the concepts of numerical range, angular numerical range, and canonical angles of a matrix \cite[Chapter 1]{horntopics}. The numerical range, also called the field of values, of a matrix $C \in \mathbb{C}^{n\times n}$ is defined as
\[
W(C) = \{ x^*Cx: x \in \mathbb{C}^n, \|x\|=1\} ,
\]
which, as a subset of $\mathbb{C}$, is compact and contains the spectrum of $C$. In addition, according to the Toeplitz-Hausdorff theorem, $W(C)$ is convex.
The angular numerical range, also called the angular field of values, of $C$ is defined as
\[
W^\prime (C) = \{ x^*Cx: x \in \mathbb{C}^n, x\neq 0 \}.
\]

If $0\notin W(C)$, then $W(C)$ is completely contained in an open half plane by its convexity. In this case, $C$ is said to be a sectorial matrix \cite{ArlinskiilPopov2003}, also called sector matrix in \cite{Lin2016}. The sectorial matrices have been widely studied in the literature \cite{Drury2013,LiSze2014,ZhangFuzhen2015}.

Let $C$ be a sectorial matrix. It is well known \cite{Horn,DeprimaJohnson1974,Drury2013} that $C$ is congruent to a unitary matrix. In particular, the unitary matrix can be chosen to be diagonal, i.e., there exist a nonsingular matrix $T$ and a diagonal unitary matrix $D$ such that
\begin{align}
C=T^*DT.\label{sectorial decomposition}
\end{align}
The factorization (\ref{sectorial decomposition}) is called sectorial decomposition \cite{ZhangFuzhen2015}. It is clear that $W^{\prime}(D)=W^{\prime}(C)$ and thus $D$ is also sectorial. Let $\delta(C)$ be the field angle of $C$, i.e., the angle subtended by the two supporting rays of $W(C)$ at the origin.
We define the phases of $C$, denoted by $\phi_1(C),\phi_2(C),\dots,\phi_n(C)$, to be the phases of the eigenvalues (i.e., diagonal elements) of $D$, taking values in an interval $(\theta, \theta+\pi)$ for some $\theta\in[-\pi,\delta(C))$. The phases defined in this way coincide with the canonical angles of $C$ introduced in \cite{FurtadoJohnson2001}.

A sectorial decomposition for a sectorial matrix is not unique. Nevertheless, the diagonal unitary matrix $D$ is unique up to a permutation, as pointed out in \cite{ZhangFuzhen2015}. As such, the phases of a sectorial matrix are uniquely defined. We adopt the convention of labeling the phases in such a way that
\[\overline{\phi}(C)=\phi_1(C)\geq \phi_2(C)\geq \dots\geq \phi_n(C)=\underline{\phi}(C).
\]
Define $\phi (C) = [ \phi_1 (C)  \ \ \phi_2 (C) \ \ \dots \ \ \phi_n(C) ]$. The uniqueness of the phases also follows directly from the maximin and minimax expressions
\cite[Lemma 8]{Horn}:
\begin{equation}
\label{pminimax}
\begin{split}
\phi_i(C)&=\max_{\mathcal{M}: \mathrm{dim}\mathcal{M}=i}\min_{x\in \mathcal{M}, \|x\|=1} \angle x^*Cx\\&=\min_{\mathcal{N}: \mathrm{dim}\mathcal{N}=n-i+1}\max_{x\in \mathcal{N}, \|x\|=1}\angle x^*Cx.
\end{split}
\end{equation}
In particular,
\begin{align*}
\overline{\phi}(C)&=\max_{x\in\mathbb{C}^n,\|x\|=1}\angle x^*Cx,\\
\underline{\phi}(C)&=\min_{x\in\mathbb{C}^n,\|x\|=1}\angle x^*Cx,
\end{align*}
from which we can observe that $\overline{\phi}(C)$ and $\underline{\phi}(C)$ enjoy good geometric interpretations.
Consider the two supporting rays of $W(C)$. Since $0 \notin  W(C)$, both supporting rays can be determined uniquely. \mbox{Figure \ref{fig1}} illustrates an example of $W(C)$ and its supporting rays. The two angles from the positive real axis to the supporting rays are $\overline{\phi}(C)$ and $\underline{\phi}(C)$ respectively. The other phases of $C$ lie in between.
\begin{figure}[htb]
\centering
\includegraphics[scale=0.5]{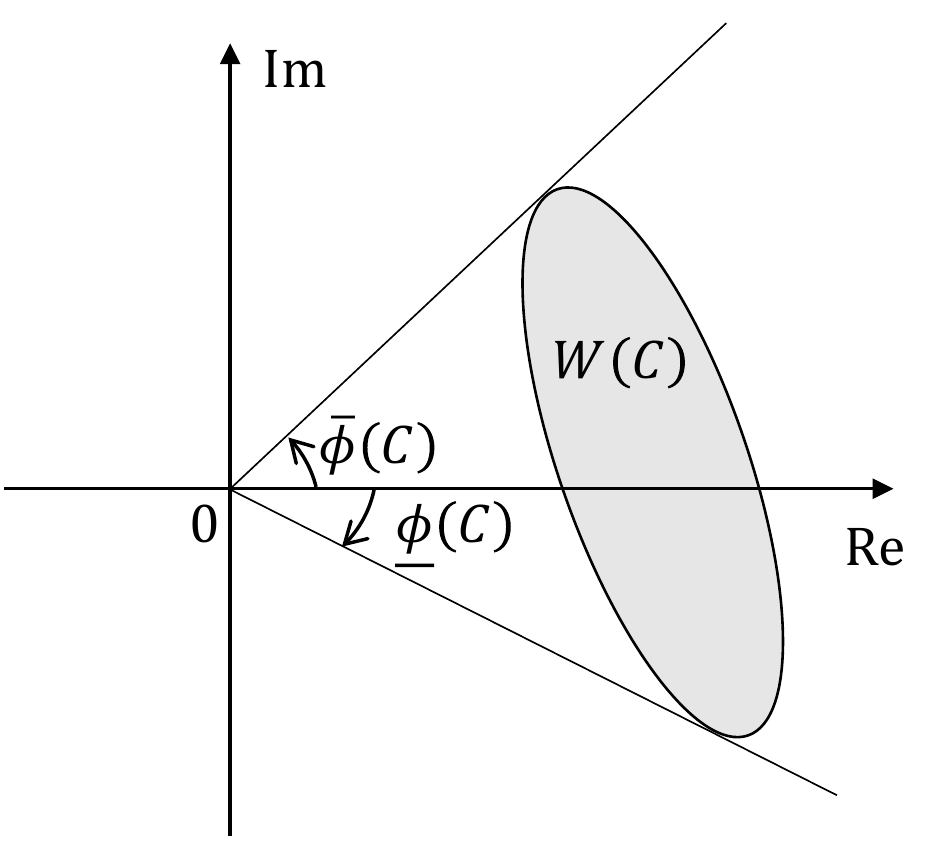}
\caption{Geometric interpretations of $\overline{\phi}(C)$ and $\underline{\phi}(C)$.}
\label{fig1}
\end{figure}

Given matrix $C$, we can check whether it is sectorial or not by plotting its numerical range. From the plot of numerical range, we can also determine a $\pi$-interval $(\theta,\theta+\pi)$ in which the phases take values. How to efficiently compute $\phi(C)$ is an important issue. The following observation provides some insights along this direction. Suppose $C$ is sectorial. Then it admits a sectorial decomposition $C=T^*DT$ and
thus
\begin{align*}
C^{-1}C^*=T^{-1}D^{-1}T^{-*}T^*D^*T=T^{-1}D^{-2}T,
\end{align*}
indicating that $C^{-1}C^*$ is similar to a diagonal unitary matrix. Hence, we can first compute $\angle\lambda(C^{-1}C^*)$, taking values in $(-2\theta\!-\!2\pi,-2\theta)$, and let $\phi(C)=-\frac{1}{2}\angle\lambda(C^{-1}C^*)$. This gives one possible way to compute $\phi(C)$. We are currently exploring other methods, hopefully of lower complexity, for determination of the interval in which matrix phases take values.



It is worth noting that the class of unitoids was introduced in \cite{JohnsonFurtado2001} to consist of matrices that are congruent to unitary matrices. Clearly, sectorial matrices constitute a special type of unitoid matrices. A nonsectorial unitoid matrix admits a factorization of the form (\ref{sectorial decomposition}). However, in this case, the eigenvalues of $D$ lie on an arc of the unit circle of length no less than $\pi$. In this paper, we do not define the phases of such matrices.

\begin{eg}
We have
\[
\phi \left( \thbth{1}{0}{0}{0}{e^{i\pi/4}}{0}{0}{0}{e^{-i\pi/4} } \right) = \begin{bmatrix}\pi/4& 0& -\pi/4\end{bmatrix}
\]
and
\[
\phi \left( \thbth{-1}{0}{0}{0}{e^{i3\pi/4}}{0}{0}{0}{e^{-i3\pi/4} } \right) = \begin{bmatrix}5\pi/4& \pi& 3\pi/4\end{bmatrix}.
\]
The matrix
\[
\thbth{1}{0}{0}{0}{e^{i2\pi/3}}{0}{0}{0}{e^{-i2\pi/3}}
\]
is not sectorial. We do not define the phases of this matrix, since in this endeavor there would clearly be an ambiguity in deciding whether the phases should be taken as $2\pi/3, 0, -2\pi/3$ or $4\pi/3, 2\pi/3, 0$.
\end{eg}

It is easy to see that the phases have the following simple properties:
\begin{enumerate}
\item The phases of a sectorial normal matrix are the phases of its eigenvalues.\label{property5}

\item The phases are invariant under congruence transformation, i.e., $\phi(C)=\phi(P^*CP)$ for every nonsingular $P$\label{property7}.

\item The phases of $C$ belong to $(- \pi/2, \pi/2)$ if and only if $C$ has a positive definite Hermitian part, i.e., $ (C+C^*)/2 > 0$. Such matrices are called strictly accretive matrices \cite{George2005} \cite[p. 279]{Kato}.\label{property8}

\end{enumerate}

If we were to use $\psi(C)$ in \eqref{eq: psi} as the definition of phases, then property \ref{property5} would be satisfied. However, properties \ref{property7}-\ref{property8} would not hold as illustrated in the following example.

\begin{eg}\label{eg: polar}
We have
\[
\psi \left(\tbt{\cos \theta}{-\sin \theta}{\sin \theta}{\cos \theta} \right) = \begin{bmatrix}\theta &-\theta\end{bmatrix}.
\]
But
\begin{multline}
\psi \left( \tbt{\cos \alpha}{0}{0}{\sin \alpha}\tbt{\cos \theta}{-\sin \theta}{\sin \theta}{\cos \theta} \tbt{\cos \alpha}{0}{0}{\sin \alpha}\right)\\= \begin{bmatrix}\arccos\displaystyle\frac{\cos \theta}{\sqrt{1-\sin^2 \theta \cos^2 2\alpha}}&\;-\arccos\displaystyle\frac{\cos \theta}{\sqrt{1-\sin^2 \theta \cos^2 2\alpha}}\end{bmatrix},\nonumber
\end{multline}
which differs from $\begin{bmatrix}\theta &-\theta\end{bmatrix}$.

We also have
\[
\psi \left(\tbt{\cos^2 \alpha}{0}{0}{\sin^2 \alpha} \tbt{\cos \theta}{-\sin \theta}{\sin \theta}{\cos \theta} \right) = \begin{bmatrix}\theta &-\theta\end{bmatrix},
\]
which is invariant of $\alpha \in (0, \pi/2)$. However, the determinant of the Hermitian part is given by
\[
\sin^2 2 \alpha \cos^2 \theta - \cos^2 2 \alpha \sin^2 \theta,
\]
which is negative for some $\theta$ close to $\pi/2$ and $\alpha$ away from $\pi/4$, such as $\theta=\pi/4$ and $\alpha= \pi/12$. That means the matrix is not strictly accretive for these values of $\theta$ and $\alpha$.
\end{eg}

The following lemma shows that the phases of a sectorial matrix provide a bound on the phases of its eigenvalues, just as the magnitudes of a matrix provide a bound on the magnitudes of its eigenvalues.
It has been proved implicitly in \cite[Lemma 9]{Horn} and also shown in \cite[Theorem 1]{FurtadoJohnson2001}.
\begin{lemma}
\label{main}
Let $C$ be sectorial with phases in $(\theta, \theta+\pi)$, where $\theta\in[-\pi, \delta(C))$. Then the phases of the eigenvalues of $C$ can be chosen so that
\[
\angle \lambda (C) \prec \phi (C).
\]
\end{lemma}

Consider the matrix
\[
C= \left( \tbt{\cos \alpha}{0}{0}{\sin \alpha}\tbt{\cos \theta}{-\sin \theta}{\sin \theta}{\cos \theta} \tbt{\cos \alpha}{0}{0}{\sin \alpha}\right)
\]
in Example \ref{eg: polar}, where $\alpha, \theta\in (0, \frac{\pi}{2})$. One can see from the discussion therein that $\psi(C) \prec \phi(C)$. In fact, this is generally true for any sectorial matrix $C$ with phases in $(\theta, \theta+\pi)$, if $\psi(C)$ takes values in the same interval, which can be inferred from \cite[Lemma 9]{Horn}.

\section{Sectorial matrix decompositions}
As discussed in the previous section, a sectorial matrix $C$ admits a sectorial decomposition as in (\ref{sectorial decomposition}), which is not unique. If
the unitary matrix in (\ref{sectorial decomposition}) is not restricted to be diagonal, one gains more flexibility. In fact, letting $T= VP$ be the unique right polar decomposition and $U=V^*DV$ leads to a new decomposition
\[
C=PUP,
\]
where $P$ is positive definite and $U$ is unitary. This is called the symmetric polar decomposition (SPD). Meanwhile, letting $T= QR$ be the unique QR factorization \cite[Theorem 2.1.14]{HornJohnson} with $Q$ being unitary and $R$ being upper triangular with positive diagonal elements and letting $W=Q^*DQ$ leads to another decomposition
\[
C=R^*WR.
\]
This is called the generalized Cholesky factorization (GCF).

\begin{theorem} \label{thm: unique}
The SPD and GCF of a sectorial matrix $C$ are unique.
\end{theorem}

\begin{proof}
We first show the uniqueness of the SPD. Suppose $C$ admits two SPDs:
\[
C = P_1U_1P_1 = P_2U_2P_2.
\]
Since the phases of a sectorial matrix are unique, we can find unitary matrices $W_1$ and $W_2$ such that $U_1=W^*_1DW_1$ and $U_2=W^*_2DW_2$, where
\[
D=\mathrm{diag}\left\{e^{i\phi_1(C)},e^{i\phi_2(C)},\dots,e^{i\phi_n(C)}\right\}.
\]
It follows that
$P_1W^*_1DW_1P_1=P_2W^*_2DW_2P_2$,
yielding
\begin{align}
D=T^*DT,\label{spd}
\end{align}
where $T=W_2P_2P^{-1}_1W^*_1$.
Suppose within the phases of $C$ there are in total $m$ distinct values $\phi_{[1]},\phi_{[2]},\dots,\phi_{[m]}$, arranged in a decreasing order with respective multiplicities given by $n_1,n_2,\dots,n_m$. Then, we can rewrite $D$ as
\[
D=\mathrm{diag}\{e^{i\phi_{[1]}}I_{n_1},e^{i\phi_{[2]}}I_{n_2},\dots,e^{i\phi_{[m]}}I_{n_m}\}
\]
and partition $T$ with compatible dimensions into
\begin{align*}
T=\begin{bmatrix}T_{11}&T_{12}&\cdots&T_{1m}\\
T_{21}&T_{22}&\cdots&T_{2m}\\
\vdots&\vdots&\ddots&\vdots\\
T_{m1}&T_{m2}&\cdots&T_{mm}\end{bmatrix},
\end{align*}
where $T_{ij}\in\mathbb{C}^{n_i\times n_j}$.

Simple computations from (\ref{spd}) yield
\begin{align}
\begin{split}
e^{i\phi_{[k]}}I_{n_k} & =\sum_{l=1}^{m}e^{i\phi_{[l]}}T^*_{lk}T_{lk}, \quad k=1,2,\dots,m\label{spd2} \quad \text{and} \\
0 & = \sum_{l=1}^m e^{i\phi_{[l]}}T_{lj}^*T_{lk}, \quad j, k = 1, 2, \dots, m, \quad j \neq k.
\end{split}
\end{align}
From the first equation in (\ref{spd2}) and the fact that all phases lie in an interval $(\theta,\theta+\pi)$, it follows that
\begin{align*}
T^*_{11}T_{11}=I_{n_1} \text{ and }T_{l1}=0,l\neq 1.
\end{align*}
Consequently, by the second equation in \eqref{spd2}, we have
\begin{align*}
T_{1k} = 0, k\neq 1.
\end{align*}
Repeated applications of the arguments above then yield
\begin{align*}
T^*_{kk}T_{kk}=I_{n_k},k=1,2,\dots,m, \text{ and }T_{lk}=0,l\neq k,
\end{align*}
which means $T$ is a block diagonal unitary matrix. Therefore, $W^*_2TW_1$ is unitary and $P_2=W^*_2TW_1P_1$ gives a polar decomposition of $P_2$. By the uniqueness of polar decomposition, we have $P_2=P_1$ and thus $U_1=U_2$.

The uniqueness of the GCF can be shown using the same lines of arguments and the uniqueness of the QR decomposition in lieu of that of the polar decomposition.
\end{proof}

For a sectorial matrix $C$ with the polar decomposition $C=\tilde{P}\tilde{U}$ and the SPD $C=PUP$, where $\tilde{P},P$ are positive definite, and $\tilde{U},U$ are unitary, we have introduced the majorization relation between $\angle\lambda(C)$, $\psi(C)$ and $\phi(C)$ in the previous section. There is an analogous relation between $|\lambda(C)|$, $\sigma(\tilde{P})$ and $\sigma(P)$. From \cite[Theorem 9.E.1]{Marshall}, we know
\[
|\lambda(C)|\prec_{\mathrm{log}}\sigma(\tilde{P}).
\]
Since $\tilde{P}=PUP\tilde{U}^*$, in view of inequality (\ref{gainmajorization}), we have
\[
\sigma(\tilde{P})\prec_{\mathrm{log}}\sigma^2(P),
\]
where $\sigma^2(P)$ is the elementwise square of $\sigma(P)$.

For a nonsectorial unitoid matrix, one can also write its SPD and GCF as above. However, neither SPD nor GCF is unique in this case.
\begin{eg}
Consider the unitoid matrix $\begin{bmatrix}0&1\\1&0\end{bmatrix}$. For any $a>0$,
\begin{align*}
\begin{bmatrix}a&0\\0&\frac{1}{a}\end{bmatrix}\begin{bmatrix}0&1\\1&0\end{bmatrix}\begin{bmatrix}a&0\\0&\frac{1}{a}\end{bmatrix} = \begin{bmatrix}0&1\\1&0\end{bmatrix}
\end{align*}
is both an SPD and a GCF.

\end{eg}

In the case where $C$ is a real sectorial matrix, sectorial decomposition, SPD, and GCF have their respective real counterparts. To see this, note that the numerical range of a real sectorial matrix is symmetric about the real axis. Let
\[
H = \frac{1}{2}(C + C^*) \quad \text{and} \quad S = \frac{1}{2}(C - C^*).
\]
Without loss of generality, suppose that $H$ is positive definite (otherwise we could work with $-C$). Let $H^{-\frac{1}{2}}$ be the inverse of the unique positive definite square root of $H$. Then
\[
H^{-\frac{1}{2}} C H^{-\frac{1}{2}} = H^{-\frac{1}{2}}(H + S)H^{-\frac{1}{2}} = I + H^{-\frac{1}{2}}SH^{-\frac{1}{2}},
\]
which is a real normal matrix. In other words, $C$ is congruent to a real normal matrix via a real congruence. Since a real normal matrix $A$ can be decomposed \cite[Theorem 2.5.8]{HornJohnson} into
\[
A = Q^* \mathrm{diag}\{A_1, \dots, A_k\} Q, \quad 1 \leq k \leq n,
\]
where $Q$ is a real orthogonal matrix and $A_i$ is either a real scalar or a real 2-by-2 matrix of the form
\[
A_i = \begin{bmatrix} \alpha_i & \beta_i \\ -\beta_i & \alpha_i \end{bmatrix},
\]
it follows that a real sectorial matrix $C$ can be factorized as
\[
C = T^* D T,
\]
where $T$ is a nonsingular real matrix and $D$ is a real block-diagonal orthogonal matrix with each block being either a scalar or a 2-by-2 matrix. We call this a real sectorial decomposition, which is nonunique in general. By performing the real polar decomposition and real QR decomposition of $T$, we arrive, respectively, at the real SPD and GCF, whose uniqueness follows from Theorem~\ref{thm: unique}. To be specific, there exist unique real positive definite $P$, real upper triangular $R$ with positive diagonal elements, and real orthogonal $U$ and $W$ such that
\[
C = PUP = R^*WR.
\]

\section{Phases of compressions and Schur complements}\label{sec:compression}

There is an interlacing relation between the magnitudes of a matrix $C$ and those of its $(n-k) \times (n-k)$ submatrices \cite[Corollary III.1.5]{Bhatia}:
\[
\sigma_i (C) \geq \sigma_i (U^*CV) \geq \sigma_{i+k} (C), \quad i=1,\dots,n-k,
\]
for all $n\times (n-k)$ isometries $U$ and $V$, which are matrices with orthonormal columns. In particular, when $k=1$, we have
\[
\sigma_1 (C) \geq \sigma_1 (U^*CV) \geq \sigma_2 (C) \geq \dots \geq \sigma_{n-1} (U^*CV) \geq \sigma_n (C).
\]

Let $C\in\mathbb{C}^{n\times n}$ be partitioned as
$C=\begin{bmatrix}C_{11}&C_{12}\\C_{21}&C_{22}\end{bmatrix}$, where $C_{11}\in\mathbb{C}^{k\times k}$. If $C_{11}$ is nonsingular, then the Schur complement \cite{ZhangFuzhen2006Schur} of $C_{11}$ in $C$, denoted by $C\slash_{11}$, exists and is given by $C\slash_{11}=C_{22}-C_{21}C_{11}^{-1}C_{12}$.
It can be inferred from \cite[Corollary 2.3]{ZhangFuzhen2006Schur} that $\sigma(C\slash_{11})$ and $\sigma(C)$ satisfy the following interlacing relation:
\[
\sigma_i(C) \geq \sigma_i(C\slash_{11}) \geq \sigma_{i+k}(C), \quad i=1,\dots,n-k.
\]

In this section, we introduce the interlacing properties between the phases of a sectorial matrix and those of its compressions and Schur complements.

Let $U$ be an $n \times (n-k)$ isometry, then $\tilde{C} = U^* C U$ is said to be a compression of
$C$. The phases of $\tilde{C}$ and those of $C$ have the following interlacing property~\cite{FurtadoJohnson2003}.

\begin{lemma}\label{thm: compression}
Let $C\in\mathbb{C}^{n\times n}$ be sectorial with phases in $(\theta, \theta+\pi)$, $\theta\in[-\pi,\delta(C))$, and $\tilde{C}\!\in\!\mathbb{C}^{(n-k)\times (n-k)}$ be a compression of $C$. Then $\tilde{C}$ is also sectorial and
\begin{equation} \label{eq: compression}
\phi_j(C) \geq \phi_j(\tilde{C}) \geq \phi_{j+k}(C), \text{ for }j=1,\dots,n-k.
\end{equation}
\end{lemma}



Note that in the special case where $k=1$, i.e., $\tilde{C}\in \mathbb{C}^{(n-1)\times (n-1)}$, the inequality (\ref{eq: compression}) becomes
\[
\phi_1 (C) \geq \phi_1 (\tilde{C}) \geq \phi_2 (C) \geq \dots \geq \phi_{n-1} (\tilde{C}) \geq \phi_n (C).
\]

\begin{remark} \label{rem: compression}
Since any full-rank $X \!\in\! \mathbb{C}^{n \times\! (n-k)}$ can be QR-decomposed as $X \!=\! QR$, where $Q \in \mathbb{C}^{n \times (n-k)}$ is an isometry and $R \in \mathbb{C}^{(n-k) \times (n-k)}$ is a nonsingular matrix, it follows from Lemma~\ref{thm: compression} that for any sectorial $C \in \mathbb{C}^{n \times n}$ and $j=1,\dots,n-k$,
\[
\phi_j(C) \geq \phi_j(X^*CX) = \phi_j(Q^*CQ) \geq \phi_{j+k}(C).
\]
\end{remark}

By exploiting the interlacing property of the phases of compressions of a sectorial matrix, we can deduce an interlacing property of the phases of the Schur complements of a sectorial matrix.
Let
$C=\begin{bmatrix}C_{11}&C_{12}\\C_{21}&C_{22}\end{bmatrix}\in\mathbb{C}^{n\times n}$ be sectorial, where $C_{11}\in\mathbb{C}^{k\times k}$. In light of Lemma~\ref{thm: compression}, $C_{11}$ is sectorial and hence nonsingular.

\begin{theorem}
Let $C\in\mathbb{C}^{n\times n}$ be sectorial with phases in $(\theta,\theta+\pi)$, $\theta\in[-\pi,\delta(C))$. Then $C\slash_{11}$ is also sectorial and
\begin{align*}\phi_j(C)\geq \phi_j(C\slash_{11})\geq \phi_{j+k}(C),
\text{ for }j=1,\dots,n-k.
\end{align*}
\end{theorem}

\begin{proof}
Let $C=PUP$ be the SPD of $C$. Then
\begin{align*}
C^{-1}=P^{-1}U^{-1}P^{-1}.
\end{align*}
Hence, $C^{-1}$ is sectorial. Moreover, if $\phi(C)$ takes values in the interval $(\theta,\theta+\pi)$, then $\phi(C^{-1})=[-\phi_n(C)\ \dots\ -\phi_1(C)]$, taking values in the interval $(-\theta\!-\!\pi,-\theta)$.

We partition $C^{-1}$ into $C^{-1}=\begin{bmatrix}(C^{-1})_{11}&(C^{-1})_{12}\\(C^{-1})_{21}&(C^{-1})_{22}\end{bmatrix}$. Then $C\slash_{11}^{-1}=(C^{-1})_{22}$. By Lemma~\ref{thm: compression}, $(C^{-1})_{22}$ is sectorial and thus, $C\slash_{11}^{-1}$ is sectorial and so is $C\slash_{11}$.
In view of (\ref{eq: compression}), we have
\begin{align*}
\phi_j(C^{-1})\geq \phi_j\left((C^{-1})_{22}\right)=\phi_j(C\slash_{11}^{-1})\geq \phi_{j+k}(C^{-1}),\; j=1,\dots,n-k.
\end{align*}
Since
\begin{align*}
\phi_j(C^{-1})&=-\phi_{n-j+1}(C),\\
\phi_j(C\slash_{11}^{-1})&=-\phi_{n-k-j+1}(C\slash_{11}),
\end{align*}
it follows that
\begin{align*}
\phi_j(C)\geq \phi_j(C\slash_{11})\geq \phi_{j+k}(C), \;j=1,\dots,n-k.
\end{align*}
\end{proof}

Lemma~\ref{thm: compression} can also be used to construct a simple proof of the following result, as a special case of \cite[Theorem 3.11]{Donnell}, which will be used in the subsequent sections.

\begin{lemma} \label{lem: extremal_rep}
Let $C\! \in \!\mathbb{C}^{n \times n}$ be sectorial with phases in $(\theta,\theta+\pi)$, where $\theta \in [-\pi,\delta(C))$. Then
\begin{align}
\max_{X \in \mathbb{C}^{n \times k} \text{ is full rank}} \sum_{i=1}^k \phi_i(X^* C X) &= \sum_{i=1}^k \phi_i(C),\label{lemma4.4-1}\\
\min_{X \in \mathbb{C}^{n \times k} \text{ is full rank}} \sum_{i=1}^k \phi_i(X^* C X) &= \sum_{i=n-k+1}^n \phi_i(C).\label{lemma4.4-2}
\end{align}
\end{lemma}

\begin{proof}
Application of Lemma~\ref{thm: compression} and Remark~\ref{rem: compression} yields
\[
\sum_{i=1}^k \phi_i(X^* C X) \leq \sum_{i=1}^k \phi_i(C)
\]
for any full-rank matrix $X \!\in\! \mathbb{C}^{n \times k}$. Let $C = T^*DT$ be a sectorial decomposition. Then defining the full-rank $X = T^{-1}\begin{bmatrix} I_k & 0 \end{bmatrix}^*$ gives $\sum_{i=1}^k \phi_i(X^* C X) = \sum_{i=1}^k \phi_i(C)$, which proves (\ref{lemma4.4-1}). The validity of (\ref{lemma4.4-2}) can be shown similarly.
\end{proof}

\section{Compound numerical ranges and numerical ranges of compounds}
\label{sec:compound}
This section explores a useful notion known as the compound numerical range. It is crucial to deriving the majorization result in the succeeding section on products of matrices. While the compound numerical range is related to the numerical range of a compound matrix, we elucidate below that the former is more informative from the perspective of characterizing matrix phases.

Let $A \in \mathbb{C}^{n \times m}$ and $1 \leq k \leq \min\{n, m\}$. The \emph{$k$th compound} of $A$, denoted by $A_{(k)}$, is the $\left(\begin{smallmatrix}
n \\
k
\end{smallmatrix}\right)
\times \left(\begin{smallmatrix}
m \\
k
\end{smallmatrix}\right)$ matrix whose elements are
\begin{align*}
\det \left(A \left[\begin{array}{lll}
i_1,&\dots,&i_k \\
j_1,&\dots,&j_k
\end{array}\right]\right)
\end{align*}
arranged lexicographically, where
\begin{align*}
A \left[\begin{array}{lll}
i_1,&\dots,&i_k \\
j_1,&\dots,&j_k
\end{array}\right]
\end{align*}
represents the submatrix of $A$ consisting of rows $i_1,\dots,i_k$ and columns $j_1,\dots,j_k$.
It is known that the eigenvalues of the $k$th compound of a square $A$ are products ($k$ at a time) of the eigenvalues of $A$. In particular, $A_{(1)} = A$ and when $A\in \mathbb{C}^{n\times n}$,
$A_{(n)} = \det A$. Moreover, if $A$ is Hermitian (resp. unitary), $A_{(k)}$ is also Hermitian (resp. unitary). It also holds that $(AB)_{(k)} = A_{(k)}B_{(k)}$ by the
Binet-Cauchy theorem. For more details on compound matrices, we refer to \cite[Chapter 19]{Marshall} and \cite[Chapter~6]{Fiedler2008book}.

Given $A \in \mathbb{C}^{n \times n}$ and $k=1,\dots,n$, we define the $k$th compound spectrum as
\[
\Lambda_{(k)}(A) = \left\{\prod_{m=1}^k \lambda_{i_m}(A): 1 \leq i_1 < \dots < i_k \leq n \right\}.
\]
It follows that
$\Lambda_{(k)}(A) \subset W(A_{(k)})$, the numerical range of the compound matrix $A_{(k)}$. Note that even if $A$ is sectorial, its compound $A_{(k)}$ is not necessarily sectorial, which means we can have the origin lying in the interior of $W(A_{(k)})$ for a sectorial $A$. This shows that the numerical range of a compound matrix is not generally suitable for characterizing compound spectra from an angular point of view.

For $k=1,\dots,n$, define the $k$th compound numerical range of $A\in\mathbb{C}^{n\times n}$ as
\[
W_{(k)}(A) = \left\{\prod_{m=1}^k \lambda_m(X^*AX): X \in \mathbb{C}^{n \times k}
\text{ is isometric} \right\},
\]
and the $k$th compound angular numerical range of $A$ as
\[
W^{\prime}_{(k)}(A) = \left\{\prod_{m=1}^k \lambda_m(X^*AX): X \in \mathbb{C}^{n \times k} \text{ is full rank} \right\}.
\]
When $k = 1$, the above definitions specialize to the familiar notions of numerical range and angular numerical range, respectively. It is straightforward to see that if $0\notin W(A)$, then $0 \notin W_{(k)}(A)$. Furthermore, $W_{(k)}(A)$ is always compact, but is not convex in general.

The following lemma generalizes \cite[Theorem 1.7.6]{horntopics}, which covers the special case of $k=1$.

\begin{lemma} \label{lem: prod_inv}
Let $A, B \in \mathbb{C}^{n \times n}$ for which $B$ is sectorial. Then
\[
\Lambda_{(k)}(AB^{-1}) \subset W_{(k)}(A)/W_{(k)}(B).
\]
\end{lemma}

\begin{proof}
First consider the case where $AB^{-1}$ is diagonalizable. By the definition of eigenvectors and eigenvalues, there exists a full-rank $X=\begin{bmatrix}x_{i_1}&x_{i_2}&\dots&x_{i_k}\end{bmatrix}\in\mathbb{C}^{n\times k}$ such that
\[
X^* AB^{-1} = \mathrm{diag}\{\lambda_{i_1}(AB^{-1}),\lambda_{i_2}(AB^{-1}),\dots,\lambda_{i_k}(AB^{-1})\}X^*
\]
for $1 \leq i_1 < \dots < i_k \leq n$. Factorize $X$ as $X = UP$ by applying the polar decomposition, where $U \in \mathbb{C}^{n \times k}$ is isometric and $P \in \mathbb{C}^{k \times k}$ is positive definite. Consequently,
\[
U^* AB^{-1} = P^{-1}\mathrm{diag}\{\lambda_{i_1}(AB^{-1}),\lambda_{i_2}(AB^{-1}),\dots,\lambda_{i_k}(AB^{-1})\}PU^*.
\]
Post-multiplying both sides by $B U$ and taking the determinants yields
\[
\prod_{i=1}^k \lambda_i(U^* A U)= \prod_{m=1}^k \lambda_{i_m} (AB^{-1}) \prod_{i=1}^k \lambda_i(U^*B U).
\]
The claim then follows by dividing both sides by $\prod_{i=1}^k \lambda_i(U^*B U)$.

When $AB^{-1}$ is not diagonalizable, choose a sequence $\{A_i\}$ with limit $A$ such that $A_iB^{-1}$ is diagonalizable for all $i = 1, 2, \dots$. It follows from the arguments above that
\[
\Lambda_{(k)}(A_iB^{-1}) \subset W_{(k)}(A_i)/W_{(k)}(B) \text{ for all }i = 1, 2, \dots.
\]
Given an isometric $X\in\mathbb{C}^{n\times k}$, we have $\prod_{m=1}^k \lambda_m(X^*A_iX)\to \prod_{m=1}^k \lambda_m(X^*AX)$ by the continuity of eigenvalues. Correspondingly, $W_{(k)}(A_i) \to W_{(k)}(A)$ in the Hausdorff metric. Similarly, we have $\Lambda_{(k)}(A_iB^{-1}) \!\to\! \Lambda_{(k)}(AB^{-1})$. Thus the desired claim follows.
\end{proof}

The next result demonstrates that the compound numerical range provides a tighter characterization of the compound spectrum than the numerical range of a compound matrix.

\begin{corollary}\label{cor: spectrum}
$\Lambda_{(k)}(A) \subset W_{(k)}(A) \subset W(A_{(k)})$.
\end{corollary}

\begin{proof}
Letting $B = I$ in Lemma~\ref{lem: prod_inv} and noting that $W_{(k)}(I) = \{1\}$ yields the first inclusion. The second inclusion follows from the fact that for any isometric $X \in \mathbb{C}^{n \times k}$,
\[
\det(X^*AX) = (X^*AX)_{(k)} = X_{(k)}^*A_{(k)}X_{(k)}
\]
and $X_{(k)}$ is a normalized column vector, since
\[
X^*X = I \implies X^*_{(k)}X_{(k)} = I_{(k)} = 1.
\]
\end{proof}

The following lemma generalizes \cite[Theorem 1.7.8]{horntopics}, which covers the special case of $k=1$.
\begin{lemma} \label{lem: prod_range}
Let $A, B \in \mathbb{C}^{n \times n}$ for which $B$ is sectorial. Then
\[
\Lambda_{(k)}(AB) \subset W^{\prime}_{(k)}(A)W^{\prime}_{(k)}(B).
\]
\end{lemma}

\begin{proof}
By Lemma~\ref{lem: prod_inv}, we have
\[
\Lambda_{(k)}(AB) \subset W_{(k)}(A)/W_{(k)}(B^{-1}).
\]
We show below that
\[
1/W_{(k)}(B^{-1}) \subset W^{\prime}_{(k)}(B),
\]
from which the claimed result follows. To this end, let
$c \in 1/W_{(k)}(B^{-1})$. This means for some isometric $X \in \mathbb{C}^{n \times k}$, we have
\[
c=\prod_{m=1}^k \frac{1}{\lambda_m(X^*B^{-1}X)}= \prod_{m=1}^k\frac{1}{\lambda_m(X^*B^{-*} B^* B^{-1}X)} = \prod_{m=1}^k \frac{1}{\lambda_m(Y^* B^* Y)},
\]
where $Y = B^{-1}X$ is full rank. Noting that $c=|c|^2/\bar{c}$, we have
\begin{align*}
c=|c|^2 \prod_{m=1}^k \lambda_m(Y^*B^*Y)^*=|c|^2 \prod_{m=1}^k \lambda_m(Y^*BY) \in W^{\prime}_{(k)}(B),
\end{align*}
as desired.
\end{proof}

\section{Phases of matrix product}
The phases of an individual complex matrix and their properties have been the focus in Sections~\ref{phase def} to~\ref{sec:compression}. Here, we study the product of two sectorial matrices $A$ and $B$.

In view of the majorization result on the magnitudes of the product of two matrices as in (\ref{gainmajorization}), it would be desirable to have a phase counterpart of the form
\[
\phi (AB) \prec \phi(A)+\phi(B)
\]
for sectorial matrices $A, B$. We know that this is true when $A,B$ are unitary \cite{nudelman1958,Thompson}. Unfortunately, this fails to hold in general, as shown in the following example.

\begin{eg}
Let $A, B$ be $n \times n$ positive definite matrices, then $\phi (A)=\phi(B)=0$. However,
$AB$ is in general not positive definite, and hence $\phi(AB)\neq0$.
\end{eg}

Notwithstanding, the following weaker but very useful result can be derived.

\begin{theorem} \label{thm: product_majorization}
Let $A,B\in\mathbb{C}^{n\times n}$ be sectorial matrices with phases in $(\theta_1, \theta_1+\pi)$ and $(\theta_2,\theta_2+\pi)$, respectively, where $\theta_1\in [-\pi,\delta(A))$ and $\theta_2\in[-\pi,\delta(B))$. Let $\angle \lambda(AB)$ take values in $(\theta_1+\theta_2, \theta_1+\theta_2+2\pi)$. Then
\begin{align}
\label{thm:main result}
\angle\lambda(AB) \prec \phi(A) + \phi(B).
\end{align}
\end{theorem}

\begin{proof}
Let $\hat{A}\!=\!e^{i(-\frac{\pi}{2}-\theta_1)}A$ and $\hat{B}\!=\!e^{i(-\frac{\pi}{2}-\theta_2)}B$. Then both $\hat{A}$ and $\hat{B}$ are sectorial matrices with phases in $(-\frac{\pi}{2},\frac{\pi}{2})$ and $\angle \lambda(\hat{A}\hat{B})$ takes values in $(-\pi,\pi)$. Label the eigenvalues of $\hat{A}\hat{B}$ so that $\angle\lambda_1(\hat{A}\hat{B}) \geq \angle\lambda_2(\hat{A}\hat{B})\geq \dots \geq \angle\lambda_n(\hat{A}\hat{B})$. Since $\phi_i(\hat{A})=\phi_i(A)-\frac{\pi}{2}-\theta_1$, $\phi_i(\hat{B})=\phi_i(B)-\frac{\pi}{2}-\theta_2$ and $\angle\lambda_i(\hat{A}\hat{B})=\angle\lambda_i(AB)-\pi-\theta_1-\theta_2$ for $i=1,\dots,n$, the inequality (\ref{thm:main result}) holds if and only if
\begin{align*}
\angle\lambda(\hat{A}\hat{B}) \prec \phi(\hat{A}) + \phi(\hat{B}).
\end{align*}

To show the above inequality, we note that $\prod_{i=1}^k\lambda_i(\hat{A}\hat{B})\!\in\!\Lambda_{(k)}(\hat{A}\hat{B})$, $k=1,\dots,n$. In view of Lemma~\ref{lem: prod_range}, it follows that
$\prod_{i=1}^k \lambda_i(\hat{A}\hat{B})\!\in\! W^{\prime}_{(k)}(\hat{A})W^{\prime}_{(k)}(\hat{B})$.
This means there exist full-rank matrices $X,Y\!\in\!\mathbb{C}^{n\times k}$ such that
\begin{align*}
\prod_{i=1}^k \lambda_i(\hat{A}\hat{B})=\prod_{i=1}^k \lambda_i(X^*\hat{A}X)\prod_{i=1}^k \lambda_i(Y^*\hat{B}Y).
\end{align*}
Since phases of $\hat{A}$ and $\hat{B}$ are in $(-\frac{\pi}{2}, \frac{\pi}{2})$ and $\angle \lambda(\hat{A}\hat{B})$ takes values in $(-\pi,\pi)$, we have
\begin{align}
\sum_{i=1}^k\angle\lambda_i(\hat{A}\hat{B})&=\sum_{i=1}^k \angle\lambda_i(X^*\hat{A}X)+\sum_{i=1}^k \angle\lambda_i(Y^*\hat{B}Y)\nonumber\\
&=\sum_{i=1}^k \phi_i(X^*\hat{A}X)+\sum_{i=1}^k \phi_i(Y^*\hat{B}Y)\nonumber\\
&\leq \sum_{i=1}^k\phi_i(\hat{A})+\sum_{i=1}^k\phi_i(\hat{B}),\label{lastine}
\end{align}
where the inequality (\ref{lastine}) is due to Lemma~\ref{lem: extremal_rep}. Note that when $k=n$, the inequality (\ref{lastine}) becomes equality, as matrix phases are invariant under congruence transformation. This completes the proof.
\end{proof}


For any $A, B \in \mathbb{C}^{n \times n}$, it is known that
\[
\sigma_i(AB) \leq \sigma_i(A)\sigma_1(B) \quad \text{for } i = 1, \dots, n.
\]
This is a special form of the more general Lidskii-Wielandt inequality, which can be found in \cite{Bhatia2001survey}. While it would be desirable to have a counterpart to this inequality for sectorial $A$ and $B$ in the form of
\[
\angle \lambda_i(AB) \leq \phi_i(A) + \phi_1(B) \quad \text{for } i = 1, \dots, n,
\]
this is generally not true. For a simple counterexample, let $B = I$ and note that
\[
\angle \lambda_i(A) \leq \phi_i(A) \quad \text{for } i = 1, \dots, n
\]
does not hold in general.

\section{Phases of matrix sum}

Given $\alpha, \beta$ for which $0\leq \beta -\alpha < 2\pi$, define
\[
\mathcal{C}[\alpha, \beta]=\left\{C\in \mathbb{C}^{n\times n}: C \text{ is sectorial and }\phi_1(C)\!\leq\! \beta,\ \phi_n(C)\!\geq\! \alpha\right\}.
\]
Clearly, $\mathcal{C}[\alpha, \beta]$ is a cone. The following theorem can be inferred from the discussions on \cite[p.~2]{ZhangFuzhen2015}. A proof is presented here for completeness.

\begin{theorem} \label{thm: matrix_sum}
Let $A, B \in \mathcal{C}[\alpha, \beta]$ with $\beta\!-\!\alpha<\pi$. Then $A+B \in \mathcal{C}[\alpha, \beta]$.
\end{theorem}

\begin{proof}
Since $A, B \in \mathcal{C}[\alpha, \beta]$, $\beta-\alpha<\pi$, there exists an open half plane containing both $W(A)$ and $W(B)$. In view of the subadditivity of numerical range \cite[Property 1.2.7]{horntopics}, i.e., $W(A+B)\subset W(A)+W(B)$ , $W(A+B)$ is contained in the same open half plane. Thus, $A+B$ is sectorial. Moreover,
\begin{align}
\phi_1(A+B)&=\max_{x\in\mathbb{C}^n,\|x\|=1}\angle(x^*Ax+x^*Bx)\nonumber\\
&\leq \max_{x\in\mathbb{C}^n,\|x\|=1}\max (\angle x^*Ax,\angle x^*Bx)\label{ine1}\\
&=\max (\phi_1(A),\phi_1(B))\nonumber\\
&\leq \beta,\nonumber\\
\phi_n(A+B)&=\min_{x\in\mathbb{C}^n,\|x\|=1}\angle(x^*Ax+x^*Bx)\nonumber\\
&\geq \min_{x\in\mathbb{C}^n,\|x\|=1}\min (\angle x^*Ax,\angle x^*Bx)\label{ine2}\\
&=\min (\phi_n(A),\phi_n(B))\nonumber\\
&\geq \alpha,\nonumber
\end{align}
where the inequalities (\ref{ine1}) and (\ref{ine2}) follow from property (\ref{boundangle}) of complex scalars. The proof is complete.
\end{proof}

This theorem says that when $\beta-\alpha<\pi$, $\mathcal{C}[\alpha, \beta]$ is a convex cone. A special case is given by $\mathcal{C}[0, 0]$, i.e., the cone of positive definite matrices. 

Theorem~\ref{thm: matrix_sum} provides an upper and a lower bound on the phases of $A+B$. It would be interesting to explore for more results regarding $\phi_i(A+B)$, such as a majorization property similar to Theorem~\ref{thm: product_majorization} or involving other geometric descriptions of $\phi_i(A+B)$.

\section{Rank robustness against perturbations}

The rank of $I + AB$ plays an important role in the field of systems and control~\cite{Zhou}. It is straightforward to see that if $\sigma(A)$ and $\sigma(B)$ are sufficiently small, then $I+AB$ has full rank. If the magnitudes of $A$ are known in advance, a problem well studied \cite{Zhou,Green} in the field of robust control is how large the magnitudes of $B$ need to be before $I+AB$ loses rank by 1, by 2, ..., or by $k$.
It can be inferred from the Schmidt-Mirsky theorem \cite[Theorem 4.18]{StewartSun1990} that for a given matrix $A \in \mathbb{C}^{n\times n}$,
\begin{equation}
\label{magnitudes-perturbation}
\inf \left\{ \overline{\sigma}(B): \mathrm{rank}(I+AB) \leq n-k \right\} = \sigma_k(A)^{-1}.
\end{equation}
Define $\mathcal{B}[\gamma]=\{C\in\mathbb{C}^{n\times n}: \overline{\sigma}(C)\leq \gamma \}$ to be the set of matrices whose largest singular values are bounded above by $\gamma>0$. Then (\ref{magnitudes-perturbation}) is equivalent to that $\mathrm{rank}(I+AB) > n-k$ for all $B\in\mathcal{B}[\gamma]$ if and only if $\gamma<{\sigma_k(A)}^{-1}$.

On the other hand, intuitively we can see that if $\phi (A)$ and $\phi (B)$ are sufficiently small in magnitudes, then $I+AB$ has full rank. This motivates us to establish a phase counterpart to the analysis of rank robustness.

For $\alpha\in[0,k\pi)$, define
\[\mathcal{C}_k[\alpha] = \left\{C\in\mathbb{C}^{n\times n}: C \text{ is sectorial and }{\sum_{i=1}^k \phi_i(C)\leq \alpha},\sum_{i=n-k+1}^n \phi_i(C)\geq -\alpha\right\}.
\]
This set is a cone but possibly non-convex unless $k=1$ and $\alpha< \pi/2$. Clearly, $\mathcal{C}_1[\alpha]$ is simply $\mathcal{C}[-\alpha,\alpha]$. \mbox{The next} theorem is concerned with the robustness of the rank of $I+AB$ under phaseal perturbations on $B$.

\begin{theorem}\label{rankrobustness}
Let $A \in \mathbb{C}^{n \times n}\!$ be sectorial with phases in $(-\pi,\pi)$ and $k=1,\dots,n$. Then $\mathrm{rank}\,(I + AB)>n-k$ for all $B\in \mathcal{C}_k[\alpha], \ \alpha\in[0,k\pi)$,
\mbox{if and only if}
\[
\alpha < \min \Bigg\{k\pi-\sum_{i=1}^k \phi_i(A),\ k\pi+\!\!\sum_{i=n-k+1}^n \!\!\!\phi_i(A)\Bigg\}.
\]
\end{theorem}

\begin{proof}
First, label the eigenvalues of $AB$ so that $\angle\lambda_1(AB)\!\geq\!\angle\lambda_2(AB)\!\geq\! \dots \geq \angle\lambda_n(AB)$. Since $A$ and $B$ are sectorial, $I + AB$ loses rank by $k$ only if
\[
\angle\lambda_1(AB) = \dots = \angle\lambda_k(AB) = \pi \quad \text{or} \quad
\angle\lambda_{n-k+1}(AB) = \dots = \angle\lambda_n(AB) = -\pi.
\]
Sufficiency thus follows from Theorem~\ref{thm: product_majorization}, whereby
\[
\sum_{i=1}^k \angle\lambda_i(AB) \leq \sum_{i=1}^k \phi_i(A) + \phi_i(B) < k\pi
\]
and
\[
\sum_{i=n-k+1}^n \angle\lambda_i(AB) \geq \sum_{i=n-k+1}^n \phi_i(A) + \phi_i(B) > -k\pi.
\]

For necessity, suppose to the contraposition that $\sum_{i=1}^k \phi_i(A) + \alpha \geq k\pi$. We construct below a matrix $B\in\mathcal{C}_k[\alpha]$ such that the rank of $I + AB$ is $n-k$. Let $A = T^*DT$ be a sectorial decomposition. Define $B = T^{-1} E T^{-*}$, where $E = \mathrm{diag}\{e_1, e_2, \dots, e_n\}$ satisfies
\begin{align*}
|e_i| & = 1 \qquad\;\; \text{for } i = 1, \dots, k, \\
\phi_i(A) + \angle e_i & = \pi \qquad\;\, \text{for } i = 1, \dots, k,  \quad \text{and}\\
e_i & =1 \qquad\;\; \text{for } i = k+1, \dots, n.
\end{align*}
Clearly, $B$ is sectorial. Moreover, $\displaystyle\sum_{i=1}^k \phi_i(B) = \displaystyle\sum_{i=1}^k \angle e_i \leq \alpha$, $\displaystyle\sum_{i=n-k+1}^n \phi_i(B)\geq -\alpha$ and
\[
AB = T^*DT T^{-1} E T^{-*} = T^* DE T^{-*}
\]
has $k$ eigenvalues at $-1$, whereby the rank of $I + AB$ is $n-k$.

The necessity of $\sum_{i=n-k+1}^n \phi_i(A) -\alpha > -k\pi$ can be shown similarly.
\end{proof}

The specialization of Theorem \ref{rankrobustness} to the case where $k=1$ is of particular importance in the study of robust control. Specifically, the theorem says that for a sectorial matrix $A \in \mathbb{C}^{n\times n}$ with phases in $(-\pi,\pi)$ and $\alpha\in[0,\pi)$, there holds
$
\mathrm{rank}\,(I+AB) =n
$
for all $B\in \mathcal{C}_1[\alpha]$ if and only if
$\alpha< \min \{\pi-\phi_1(A), \pi+\phi_n(A) \}$.

Combining this understanding and (\ref{magnitudes-perturbation}), we have the following result. The proof is straightforward and is thus omitted for brevity.
\begin{theorem}\label{thm: mix}
Let $A\in\mathbb{C}^{n\times n}$ be sectorial with phases in $(-\pi,\pi)$. Then
$
\mathrm{rank}\,(I+AB)=n
$
for all $B\in \mathcal{B}[\gamma]\cup \mathcal{C}_1[\alpha]$ with $\gamma>0$, $\alpha\in[0,\pi)$, if and only if $\gamma< {\sigma_1(A)}^{-1}$ and $\alpha< \min \{\pi-\phi_1(A), \pi+\phi_n(A) \}$.
\end{theorem}



\section{Sectorial matrix completion}
The matrix completion problem is to determine the unspecified entries of a partial matrix so that the completed matrix has certain desired properties.
Consider partial matrices of the form
\begin{align*}
C=\left[
\begin{matrix}
C_{11}& \cdots&C_{1,1+p}& &?\\
\vdots& & &\ddots& \\
C_{p+1,1} & & & &C_{n-p,n}\\
&\ddots& & &\vdots\\
?& &C_{n,n-p}&\cdots& C_{nn}
\end{matrix}\right],
\end{align*}
where $C_{ij}\in\mathbb{C}^{n_i\times n_j}$ is specified for $|i-j|\leq p$ with $p$ a fixed integer $0\leq p\leq n-1$. The unspecified blocks are represented by question marks. Such partial matrices are called $p$-banded. See \cite{BakonyiWoerdeman} for comprehensive discussions of matrix completions. It is shown in \cite{Grone1984} that $C$ admits a positive semidefinite completion if and only if
\begin{align*}
\left[\begin{matrix}
C_{ll} &\cdots&C_{l,l+p}\\
\vdots& &\vdots\\
C_{l+p,l}&\cdots&C_{l+p,l+p}
\end{matrix}\right]\geq0, \text{ for }l=1,\dots,n-p.
\end{align*}
In a similar spirit, we have the following result.
\begin{theorem}
A $p$-banded partial matrix $C$ admits a completion in $\mathcal{C}[\alpha,\beta]$ with $0\leq\beta-\alpha<\pi$ if and only if
\begin{align*}
\left[\begin{matrix}
C_{ll} &\cdots&C_{l,l+p}\\
\vdots& &\vdots\\
C_{l+p,l}&\cdots&C_{l+p,l+p}
\end{matrix}\right]\in\mathcal{C}[\alpha,\beta], \text{ for } l=1,\dots,n-p.
\end{align*}
\end{theorem}

\begin{proof}
The necessity follows from Lemma~\ref{thm: compression}. It remains to show the sufficiency. Without loss of generality, assume $n=3$ and $p=1$. The general case can be shown by an additional induction.

Let
\begin{align*}
C=\left[\begin{matrix}
C_{11}&C_{12}&X_{13}\\
C_{21}&C_{22}&C_{23}\\
Y_{31}&C_{32}&C_{33}
\end{matrix}\right],
\end{align*}
where $X_{13}$ and $Y_{31}$ are to be determined so that
$C\in \mathcal{C}[\alpha,\beta]$, which is equivalent to requiring the two inequalities
\begin{align}
\label{ieq: rotation1}
e^{i(\frac{\pi}{2}-\beta)}C+e^{-i(\frac{\pi}{2}-\beta)}C^*&\geq0,\\ e^{i(-\frac{\pi}{2}-\alpha)}C+e^{-i(-\frac{\pi}{2}-\alpha)}C^*&\geq0\label{ieq: rotation2},
\end{align}
hold simultaneously.
To find such $X_{13}$ and $Y_{31}$, we partition  $e^{i(\frac{\pi}{2}-\beta)}C+e^{-i(\frac{\pi}{2}-\beta)}C^*$ with compatible dimensions into
\begin{align*}
e^{i(\frac{\pi}{2}-\beta)}C+e^{-i(\frac{\pi}{2}-\beta)}C^*=\left[\begin{matrix}
A&B&X\\B^*&E&F\\X^*&F^*&G
\end{matrix}\right].
\end{align*}
When $\begin{bmatrix}C_{11}&C_{12}\\C_{21}&C_{22}\end{bmatrix},\begin{bmatrix}C_{22}&C_{23}\\C_{32}&C_{33}\end{bmatrix}\in\mathcal{C}[\alpha,\beta]$, both $\begin{bmatrix}
A&B\\B^*&E
\end{bmatrix}$ and $\begin{bmatrix}
E&F\\F^*&G
\end{bmatrix}$ are positive semidefinite.
By \cite[Theorem 3.2]{Smith2008}, if we let $X=BE^\dagger F$, i.e.,
\begin{equation}\label{inequality1}
e^{i(\frac{\pi}{2}-\beta)}X_{13}+e^{-i(\frac{\pi}{2}-\beta)}Y^*_{31}=BE^\dagger F,  \end{equation}
where $E^\dagger$ is the Moore-Penrose pseudoinverse of $E$, then the inequality (\ref{ieq: rotation1}) holds.
Similarly, we partition
\begin{align*}
e^{i(-\frac{\pi}{2}-\alpha)}C+e^{-i(-\frac{\pi}{2}-\alpha)}C^*=\left[\begin{matrix}
\tilde{A}&\tilde{B}&\tilde{X}\\\tilde{B}^*&\tilde{E}&\tilde{F}\\\tilde{X}^*&\tilde{F}^*&\tilde{G}
\end{matrix}\right]
\end{align*}
and let $\tilde{X}=\tilde{B}\tilde{E}^\dagger\tilde{F}$, i.e.,
\begin{equation}\label{inequality2}
e^{i(-\frac{\pi}{2}-\alpha)}X_{13}+e^{-i(-\frac{\pi}{2}-\alpha)}Y^*_{31}=\tilde{B}\tilde{E}^{\dagger}\tilde{F}.
\end{equation}
Then the inequality (\ref{ieq: rotation2}) holds.
Finally, solving equations (\ref{inequality1}) and (\ref{inequality2}) together yields desired $X_{13}$ and $Y_{31}$ so that $C\in \mathcal{C}[\alpha,\beta]$.
\end{proof}

\section{Sectorial matrix decomposition}
In this section, we discuss a matrix decomposition problem, which can be regarded as a dual problem of the sectorial matrix completion studied in the previous section. A $p$-banded matrix
\begin{align*}
C=\left[
\begin{matrix}
C_{11}& \cdots&C_{1,p+1}& &0\\
\vdots& & &\ddots& \\
C_{p+1,1} & & & &C_{n-p,n}\\
&\ddots& & &\vdots \\
0& &C_{n,n-p}&\cdots& C_{nn}
\end{matrix}\right],
\end{align*}
where $C_{ij}\in\mathbb{C}^{n_i\times n_j}$, is said to admit a positive semidefinite decomposition if it holds $C=C_1+\dots+C_{n-p}$ as in Figure~\ref{fig:banded decomposition}, with $C_l=\mathrm{diag}\{0, \tilde{C}_l, 0\}$ and $\tilde{C}_l\geq0$ for $l=1, \dots, n-p$. It is shown in \cite{Martin1965} that $C$ admits a positive semidefinite decomposition if and only if $C\geq0$.

\begin{figure}[htb]
\centering
\includegraphics[scale=0.4]{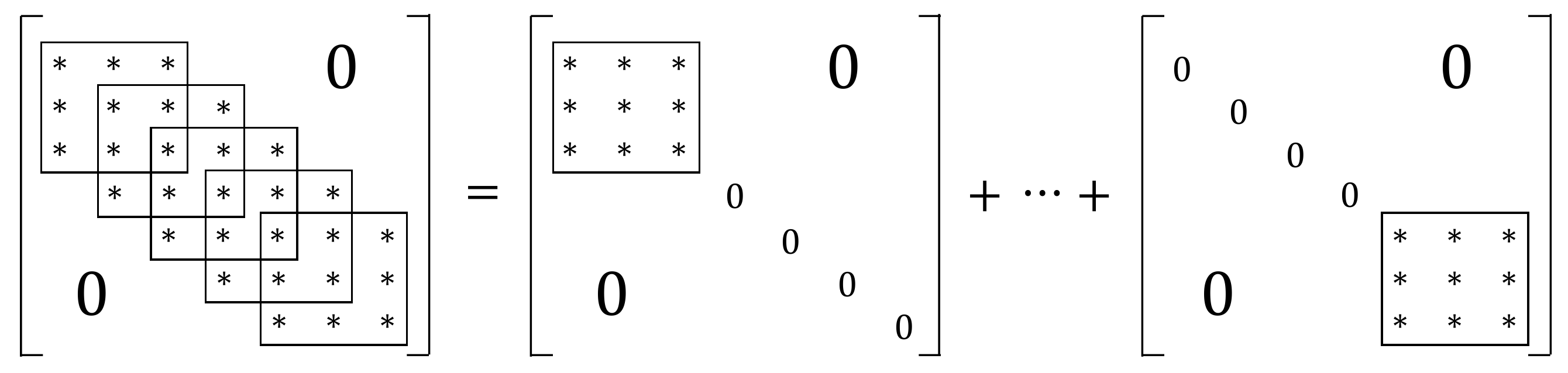}
\caption{Banded matrix decomposition.}
\label{fig:banded decomposition}
\end{figure}

Extending to the realm of sectorial matrices, we say a matrix $C$ admits a decomposition in $\mathcal{C}[\alpha,\beta]$ with $0\leq \beta-\alpha<\pi$ if it holds $C=C_1+\dots+C_{n-p}$ with $C_l=\mathrm{diag}\{0, \tilde{C}_l, 0\}$ and $\tilde{C}_l\in\mathcal{C}[\alpha,\beta]$ for $l=1, \dots, n-p$. We have the following result.

\begin{theorem}
\label{thm: banded decomposition}
A $p$-banded matrix $C$ admits a decomposition in $\mathcal{C}[\alpha,\beta]$ if and only if $C\in\mathcal{C}[\alpha,\beta]$.
\end{theorem}

\begin{proof}
We first show the necessity. From $\tilde{C}_l\in\mathcal{C}[\alpha,\beta]$, we have
\begin{align*}
e^{i(\frac{\pi}{2}-\beta)}\tilde{C}_l+e^{-i(\frac{\pi}{2}-\beta)}\tilde{C}^*_l\geq0 \text{ and }
e^{i(-\frac{\pi}{2}-\alpha)}\tilde{C}_l+e^{-i(-\frac{\pi}{2}-\alpha)}\tilde{C}^*_l\geq0.
\end{align*}
Then,
\begin{equation*}
\begin{split}
e^{i(\frac{\pi}{2}-\beta)}C+e^{-i(\frac{\pi}{2}-\beta)}C^*&=\sum_{l=1}^{n-p}\left(e^{i(\frac{\pi}{2}-\beta)}C_l+e^{-i(\frac{\pi}{2}-\beta)}C_l^*\right)\geq0,\\
e^{i(-\frac{\pi}{2}-\alpha)}C+e^{-i(-\frac{\pi}{2}-\alpha)}C^*&=\sum_{l=1}^{n-p}\left(e^{i(-\frac{\pi}{2}-\alpha)}C_l+e^{-i(-\frac{\pi}{2}-\alpha)}C_l^*\right)\geq0,
\end{split}
\end{equation*}
from which $C\in\mathcal{C}[\alpha,\beta]$ follows.

We will show the sufficiency by construction. Without loss of generality, assume $n=3$ and $p=1$. The general case can be shown by an additional induction.

Let
\begin{align*}
C=\left[\begin{matrix}
C_{11}&C_{12}&0\\
C_{21}&C_{22}&C_{23}\\
0&C_{32}&C_{33}
\end{matrix}\right]= \left[\begin{matrix}
C_{11}&C_{12}&0\\
C_{21}&X_{22}&0\\
0&0&0
\end{matrix}\right]+\left[\begin{matrix}
0&0&0\\
0&Y_{22}&C_{23}\\
0&C_{32}&C_{33}
\end{matrix}\right],
\end{align*}
where $X_{22}$ and $Y_{22}$ are to be determined so that $\begin{bmatrix}
C_{11}&C_{12}\\
C_{21}&X_{22}
\end{bmatrix}, \begin{bmatrix}
Y_{22}&C_{23}\\
C_{32}&C_{33}
\end{bmatrix}\in\mathcal{C}[\alpha,\beta]$. By $C\in\mathcal{C}[\alpha,\beta]$, we have
\begin{align}
e^{i(\frac{\pi}{2}-\beta)}C+e^{-i(\frac{\pi}{2}-\beta)}C^*&\geq0,\\
e^{i(-\frac{\pi}{2}-\alpha)}C+e^{-i(-\frac{\pi}{2}-\alpha)}C^*&\geq0\label{ineq2:decompostion}.
\end{align}
Partitioning $e^{i(\frac{\pi}{2}-\beta)}C+e^{-i(\frac{\pi}{2}-\beta)}C^*$ with compatible dimensions into
\begin{align*}
e^{i(\frac{\pi}{2}-\beta)}C+e^{-i(\frac{\pi}{2}-\beta)}C^*=\left[\begin{matrix}
A&B&0\\B^*&E&F\\0&F^*&G
\end{matrix}\right]=\left[\begin{matrix}
A&B&0\\
B^*&X&0\\
0&0&0
\end{matrix}\right]+\left[\begin{matrix}
0&0&0\\
0&Y&F\\
0&F^*&G
\end{matrix}\right],
\end{align*}
we have $\begin{bmatrix}
A&B\\B^*&E\end{bmatrix}\geq 0$, which implies $\mathcal{R}(B)\subset\mathcal{R}(A)$ \cite[Theorem 1.19]{ZhangFuzhen2006Schur}, where $\mathcal{R}$ denotes the range.
Moreover, if we set $X=B^*A^{\dagger}B$, we can see that the generalized Schur complement of $A$ in $\begin{bmatrix}
A&B\\B^*&X
\end{bmatrix}$ equals $0$. Thus, by a property of generalized Schur complement \cite[Theorem 1.20]{ZhangFuzhen2006Schur}, we have $\begin{bmatrix}A&B\\B^*&X\end{bmatrix}\geq 0$.

It remains to show that $Y\geq 0$ and $\begin{bmatrix}Y&F\\F^*&G\end{bmatrix}\geq 0$. Since $Y=E-X=E-B^*A^{\dagger}B$ is the generalized Schur complement of $A$ in $\begin{bmatrix}A&B\\B^*&E\end{bmatrix}$, we have $Y\geq 0$.
Furthermore, by
\begin{align*}
\left[\begin{matrix}
A&B&0\\B^*&E&F\\0&F^*&G
\end{matrix}\right]\geq 0,
\end{align*}
we can see
\begin{align*}
\begin{bmatrix}E&F\\F^*&G\end{bmatrix}-\begin{bmatrix}B^*\\0\end{bmatrix}A^{\dagger}\begin{bmatrix}B&0\end{bmatrix}=\begin{bmatrix}E-B^*A^{\dagger}B&F\\F^*&G\end{bmatrix}=\begin{bmatrix}Y&F\\F^*&G\end{bmatrix}\geq 0.   \end{align*}
Therefore, we have an equation of $X_{22}$:
\begin{align}\label{eq1:decomposition}
e^{i(\frac{\pi}{2}-\beta)}X_{22}+e^{-i(\frac{\pi}{2}-\beta)}X_{22}^*=B^*A^\dagger B.
\end{align}
Similarly, we partition
\begin{align*}
e^{i(-\frac{\pi}{2}-\alpha)}C+e^{-i(-\frac{\pi}{2}-\alpha)}C^*=\left[\begin{matrix}
\tilde{A}&\tilde{B}&0\\\tilde{B}^*&\tilde{E}&\tilde{F}\\0&\tilde{F}^*&\tilde{G}
\end{matrix}\right]=\left[\begin{matrix}
\tilde{A}&\tilde{B}&0\\
\tilde{B}^*&\tilde{X}&0\\
0&0&0
\end{matrix}\right]+\left[\begin{matrix}
0&0&0\\
0&\tilde{Y}&\tilde{F}\\
0&\tilde{F}^*&\tilde{G}
\end{matrix}\right]
\end{align*}
and let $\tilde{X}=\tilde{B}^*\tilde{A}^\dagger\tilde{B}$, i.e.,
\begin{align}\label{eq2:decomposition}
e^{i(-\frac{\pi}{2}-\alpha)}X_{22}+e^{-i(-\frac{\pi}{2}-\alpha)}X_{22}^*=\tilde{B}^*\tilde{A}^\dagger \tilde{B}.
\end{align}
Then $\begin{bmatrix}\tilde{A}&\tilde{B}\\\tilde{B}^*&\tilde{X}\end{bmatrix}\geq 0, \begin{bmatrix}\tilde{Y}&\tilde{F}\\\tilde{F}^*&\tilde{G}\end{bmatrix}\geq 0$. Finally, solving equations (\ref{eq1:decomposition}) and (\ref{eq2:decomposition}) together yields desired $X_{22}$.
\end{proof}

\section{Kronecker and Hadamard products}
The Kronecker product of $A\!\in\!\mathbb{C}^{n\times n}$ and $B\in\mathbb{C}^{m\times m}$, denoted by $A\otimes B$, is given by
\begin{align*}
A \otimes B = \left[
\begin{matrix}
a_{11}B & \cdots & a_{1n}B\\
\vdots & &\vdots\\
a_{n1}B &\cdots&a_{nn}B
\end{matrix}\right]\in\mathbb{C}^{nm\times nm}.
\end{align*}
It is known that the singular values of $A\otimes B$ are given by
$\sigma_i(A)\sigma_j(B), 1\leq i\leq n, 1\leq j\leq m$ \cite[Theorem 4.2.15]{horntopics}.
Regarding the phases of $A\otimes B$, we have the following result.

\begin{theorem} \label{thm: kronecker}
Let $A\in\mathbb{C}^{n\times n}$ and $B\in\mathbb{C}^{m\times m}$ be sectorial matrices with phases in $(\theta_1, \theta_1+\pi)$ and $(\theta_2,\theta_2+\pi)$, where $\theta_1\in[-\pi, \delta(A))$ and $\theta_2 \in [-\pi,\delta(B))$, respectively. If $\phi_1(A) + \phi_1(B) - \phi_n(A) - \phi_n(B) < \pi$, then $A\otimes B$ is sectorial and its phases are given by
$\phi_i(A) + \phi_j(B), 1 \leq i \leq n, 1 \leq j \leq m$.
\end{theorem}

\begin{proof}
Let $A=T^*DT$ and $B=R^*ER$ be sectorial decompositions of $A$ and $B$, respectively. Then,
\begin{align}
A\otimes B=(T^*DT)\otimes (R^*ER)=(T\otimes R)^*(D\otimes E)(T \otimes R),\label{sectorial decompositionotimes}
\end{align}
where $T\otimes R$ is nonsingular and $D\otimes E$ is diagonal unitary. Note that the eigenvalues of $D\otimes E$ are given by $\lambda_i(D)\lambda_j(E), 1 \leq i\leq n, 1\leq j \leq m$, where
\[
\angle (\lambda_i(D)\lambda_j(E))=\angle\lambda_i(D)+\angle\lambda_j(E )=\phi_i(A)+\phi_j(B).
\]
Since $\phi_1(A) + \phi_1(B) - \phi_n(A) - \phi_n(B) < \pi$, it follows that $A\otimes B$ is sectorial and (\ref{sectorial decompositionotimes}) is a sectorial decomposition of $A\otimes B$. Furthermore, the phases of $A\otimes B$ are given by $\angle(\lambda_i(D)\lambda_j(E))$, i.e., $\phi_i(A) + \phi_j(B)$, $1 \leq i\leq n, 1\leq j \leq m$.
\end{proof}


As for the Hadamard product $A\odot B$, a notably elegant result on its singular values is that \cite[Theorem 5.5.4]{horntopics}
\[
\sigma(A\odot B)\prec_w \sigma(A)\odot\sigma(B).
\]
One may expect a phase counterpart in the form of $\phi (A \odot B) \prec_w \phi(A)+\phi(B)$. The following example demonstrates that this is not true in general.

\begin{eg}
Let
\begin{align*}
A&=\left[\begin{matrix}
3 - 2i &  1- 2i  & 1 &  1 + i\\
1- 2i  & 2 &  -i &  -i\\
1 &  -i  & 1 + 3i  & 3i\\
1+i &  -i &  3i &  1 + 4i
\end{matrix}\right]\text{ and } B=I.
\end{align*}
Then
\begin{align*}
\phi(A\odot B)&=[1.3258\ \ 1.249 \hspace{12pt} 0\hspace{29pt} -0.588],\\
\phi(A)+\phi(B)&=[1.5303\ \ 0.7684\ \ 0.3561\ \ -0.7926].
\end{align*}
It can be seen that no majorization type relation holds between $\phi(A\odot B)$ and $\phi(A)+\phi(B)$.
This example also invalidates $\angle \lambda (A\odot B)\prec_w\phi(A)+\phi(B)$.
\end{eg}

Notwithstanding, we have the following weaker result.
\begin{theorem}
Let $A,B\in\mathbb{C}^{n\times n}$ be sectorial matrices with phases in $(\theta_1, \theta_1+\pi)$ and $(\theta_2,\theta_2+\pi)$, where $\theta_1\in [-\pi,\delta(A))$ and $\theta_2 \in [-\pi,\delta(B))$, respectively. If $\phi_1(A) + \phi_1(B) - \phi_n(A) - \phi_n(B) < \pi$, then $A\odot B$ is sectorial,
 \begin{align*}
 \phi_1(A \odot B) \leq \phi_1(A)+\phi_1(B) \quad \text{and} \quad
 \phi_n(A \odot B) \geq \phi_n(A)+\phi_n(B).
 \end{align*}
\end{theorem}

\begin{proof}
 The claim follows from Theorems \ref{thm: kronecker} and Lemma \ref{thm: compression}, since $A\odot B$ can be expressed as a compression of $A\otimes B$. We offer an alternative proof below.

Let $x\in\mathbb{C}^n, x\neq 0$. By \cite[Lemma 5.1.5]{horntopics},
 \[
 x^* (A \odot B) x = \mathrm{trace}(D^*_x A D_x B^T) = \sum_{i=1}^n \lambda_i(D^*_x A D_x B^T),
 \]
 where $D_x$ is a diagonal matrix with diagonal entries given by $x_i$. Since
 \[
 \lambda_i(D^*_x A D_x B^T) \in W^\prime(D^*_x A D_x)W^\prime(B^T) \subset W^\prime(A)W^\prime(B),
 \]
 we have $x^* (A \odot B) x \!\neq\! 0$ and
 $\angle x^* (A \odot B) x \!\in\! [\phi_n(A) + \phi_n(B), \phi_1(A) + \phi_1(B)]$,
 as required.
\end{proof}

\section{Conclusions}
In this paper, we define the phases of a sectorial matrix and study their properties. We introduce the symmetric polar decomposition and generalized Cholesky factorization of a sectorial matrix and establish their uniqueness. The symmetric polar decomposition seems to have advantages over the usual polar decomposition, at least in defining the matrix phases. In the scalar case, the symmetric polar decomposition takes the form $c=\sqrt{\sigma}e^{i\phi}\sqrt{\sigma}$. A number of useful properties of the matrix phases have been studied, including those of compressions, Schur complements, matrix products and sums. The rank robustness of a matrix against magnitude/phase perturbations, motivated from applications in robust control, has also been examined. A pair of problems: sectorial matrix completion and decomposition extending those of the positive semidefinite completion and decomposition are studied.

The definition of phases can be generalized to matrices whose numerical ranges contain the origin on their boundaries. We call such matrices semi-sectorial. A number of results in this paper have potential extensions to semi-sectorial matrices, including those on sectorial and symmetric polar decompositions, matrix products, sums, and rank robustness.

The definition of phases can be further extended to some matrices which are not sectorial, such as block diagonal matrices with sectorial diagonal blocks, the Kronecker products of sectorial matrices, compound sectorial matrices, etc. These matrices are constructed from sectorial matrices and their phases can be defined by exploiting the phases of the original sectorial matrices.

How to define phases for general nonsectorial unitoid matrices, which are congruent to unitary matrices, remains open. A critical issue in this problem involves determining how the phases take values and deriving their corresponding properties. We expect that the notion of Riemann surface would play an important role in studying the phases of unitoid matrices.

\section*{Acknowledgment}
This work was partially supported by the Research Grants Council of
Hong Kong Special Administrative Region, China, under the Theme-Based
Research Scheme T23-701/14-N and the General Research Fund 16211516.

The authors also wish to thank Di Zhao, Axel Ringh, and Chi-Kwong Li for valuable discussions.

\bibliographystyle{elsarticle-num}
 \bibliography{dynamic_v2}

\begin{thebibliography}{10}
\expandafter\ifx\csname url\endcsname\relax
  \def\url#1{\texttt{#1}}\fi
\expandafter\ifx\csname urlprefix\endcsname\relax\def\urlprefix{URL }\fi
\expandafter\ifx\csname href\endcsname\relax
  \def\href#1#2{#2} \def\path#1{#1}\fi

\bibitem{HornJohnson}
R.~A. Horn, C.~R. Johnson, Matrix Analysis, Cambridge University Press, 1990.

\bibitem{Bhatia}
R.~Bhatia, Matrix Analysis, Springer-Verlag, New York, 1997.

\bibitem{Marshall}
A.~W. Marshall, I.~Olkin, B.~C. Arnold, Inequalities: Theory of Majorization
  and Its Applications, Springer, 1979.

\bibitem{Macfarlane1981}
I.~Postlethwaite, J.~Edmunds, A.~MacFarlane, Principal gains and principal
  phases in the analysis of linear multivariable feedback systems, IEEE
  Transactions on Automatic Control 26~(1) (1981) 32--46.

\bibitem{FurtadoJohnson2001}
S.~Furtado, C.~R. Johnson, Spectral variation under congruence, Linear and
  Multilinear Algebra 49~(3) (2001) 243--259.

\bibitem{horntopics}
R.~A. Horn, C.~R. Johnson, Topics in Matrix Analysis, Cambridge University
  Press, 1991.

\bibitem{ArlinskiilPopov2003}
Y.~M. Arlinski{\i}, A.~B. Popov, On sectorial matrices, Linear Algebra and its
  Applications 370 (2003) 133--146.

\bibitem{Lin2016}
M.~Lin, Some inequalities for sector matrices, Operators and Matrices 10~(4)
  (2016) 915--921.

\bibitem{Drury2013}
S.~Drury, Fischer determinantal inequalities and {H}igham's conjecture, Linear
  Algebra and its Applications 439~(10) (2013) 3129--3133.

\bibitem{LiSze2014}
C.-K. Li, N.-S. Sze, Determinantal and eigenvalue inequalities for matrices
  with numerical ranges in a sector, Journal of Mathematical Analysis and
  Applications 410~(1) (2014) 487--491.

\bibitem{ZhangFuzhen2015}
F.~Zhang, A matrix decomposition and its applications, Linear and Multilinear
  Algebra 63~(10) (2015) 2033--2042.

\bibitem{Horn}
A.~Horn, R.~Steinberg, Eigenvalues of the unitary part of a matrix, Pacific
  Journal of Mathematics 9~(2) (1959) 541--550.

\bibitem{DeprimaJohnson1974}
C.~DePrima, C.~R. Johnson, The range of {$A^{-1}A^*$} in {$GL(n, C)$}, Linear
  Algebra and its Applications 9 (1974) 209--222.

\bibitem{JohnsonFurtado2001}
C.~R. Johnson, S.~Furtado, A generalization of {S}ylvester's law of inertia,
  Linear Algebra and its Applications 338~(1-3) (2001) 287--290.

\bibitem{George2005}
A.~George, K.~D. Ikramov, On the properties of accretive-dissipative matrices,
  Mathematical Notes 77~(5-6) (2005) 767--776.

\bibitem{Kato}
T.~Kato, Perturbation Theory for Linear Operators, Springer-Verlag, 1980.

\bibitem{ZhangFuzhen2006Schur}
F.~Zhang, The Schur Complement and Its Applications, Springer Science \&
  Business Media, 2006.

\bibitem{FurtadoJohnson2003}
S.~Furtado, C.~R. Johnson, Spectral variation under congruence for a
  nonsingular matrix with 0 on the boundary of its field of values, Linear
  Algebra and its Applications 359~(1-3) (2003) 67--78.

\bibitem{Donnell}
W.~A. Donnell, Minimax principles of the arguments of the proper values of a
  normal linear transformation, Ph.D. thesis, Texas Tech University (1971).

\bibitem{Fiedler2008book}
M.~Fiedler, Special Matrices and Their Applications in Numerical Mathematics,
  2nd ed., Dover Publications, Inc., 2008.

\bibitem{nudelman1958}
A.~A. Nudelman, P.~Shvartsman, The spectrum of the product of unitary matrices,
  Uspekhi Matematicheskikh Nauk 13~(6) (1958) 111--117.

\bibitem{Thompson}
R.~Thompson, On the eigenvalues of a product of unitary matrices {I}, Linear
  and Multilinear Algebra 2~(1) (1974) 13--24.

\bibitem{Bhatia2001survey}
R.~Bhatia, Linear algebra to quantum cohomology: The story of {A}lfred {H}orn's
  inequalities, The American Mathematical Monthly 108~(4) (2001) 289--318.

\bibitem{Zhou}
K.~Zhou, J.~C. Doyle, K.~Glover, Robust and Optimal Control, Upper Saddle
  River, NJ, USA: Prentice-Hall, Inc., 1996.

\bibitem{Green}
M.~Green, D.~J. Limebeer, Linear Robust Control, Englewood Cliffs, NJ, USA:
  Prentice-Hall, Inc., 1995.

\bibitem{StewartSun1990}
G.~Stewart, J.-G. Sun, Matrix Perturbation Theory, Academic Press, 1990.

\bibitem{BakonyiWoerdeman}
M.~Bakonyi, H.~J. Woerdeman, Matrix Completions, Moments, and Sums of Hermitian
  Squares, Vol.~37, Princeton University Press, 2011.

\bibitem{Grone1984}
R.~Grone, C.~R. Johnson, E.~M. S{\'a}, H.~Wolkowicz, Positive definite
  completions of partial {H}ermitian matrices, Linear Algebra and its
  Applications 58 (1984) 109--124.

\bibitem{Smith2008}
R.~L. Smith, The positive definite completion problem revisited, Linear Algebra
  and its Applications 429~(7) (2008) 1442--1452.

\bibitem{Martin1965}
R.~Martin, J.~Wilkinson, Symmetric decomposition of positive definite band
  matrices, Numerische Mathematik 7~(5) (1965) 355--361.

\end{thebibliography}

\end{document}